\documentclass{tac}

\usepackage{amsfonts}
\usepackage{bbm}
\usepackage{amsmath}
\usepackage[all,cmtip,2cell]{xy}
\usepackage{tensor}
\usepackage{adjustbox}
\usepackage{amssymb}
\usepackage{xfrac}
\usepackage{faktor}
\usepackage{tikz-cd}
\UseAllTwocells

\usepackage{xy}

\input diagxy

\usepackage[colorlinks=true]{hyperref}
\hypersetup{allcolors=[rgb]{0.1,0.1,0.4}}

\author{Phillip M Bressie}

\thanks{}

\address{Department of Mathematics and Computer Science, Spring Hill College\\
 4000 Dauphin St., Mobile, AL, 36608\\[5pt]
 }

\title {On Tautological Globular Operads}

\copyrightyear{2023}


\eaddress{pbressie@shc.edu}

\newcommand{\catname}[1]{{\normalfont\textbf{#1}}}
\newcommand{\Glob}{\catname{Glob}}
\newcommand{\Grdset}{\catname{GrdSet}}
\newcommand{\Coll}{\catname{Col}} 

\newcommand{\one}{\textbf{1}}
\newcommand{\Tone}{\mathcal{T}(\one)}
\newcommand{\Tsquared}{\mathcal{T}^2(\one)}
\newcommand{\id}{\mathbbm{1}}

\DeclareSymbolFont{AMSa}{U}{msa}{m}{n}
\DeclareMathSymbol{\mysquare}{\mathord}{AMSa}{"03}

\newcommand{\pullbackmark}[2]{\save ;p+<.8pc,0pc>:(0,-1)::%
	(#1) *{\phantom{Z}} %
	;p+(#2)-(0,0) **@{-}%
	;p-(#1)+(0,0) *{\phantom{Z}} **@{-} \restore}

\newcommand{\pullbackSub}[2]{{}_{#1}\kern-\scriptspace{\times}_{#2}}

\newtheorem{theorem}{Theorem}

\newtheoremrm{rem}{Remark}

\mathrmdef{Hom}
\mathbfdef{Set}

\begin{document}

\maketitle
\begin{abstract}
 The purpose of this exposition is to compare the constructions of classical nonsymmetric operads (and their algebras) to that of the globular operads of Leinster and Batanin.  It is hoped that, through this comparison, understanding algebras for globular operads can be made more intuitive and approachable.  We begin by giving a description of the construction of the classical tautological, or endomorphism, operad $taut(X)$ on a set $X$.  We then describe how globular operads are a strict generalization of classical operads.  From this perspective a description is given of the construction for the tautological globular operad $Taut(\mathcal{X})$ on a globular set $\mathcal{X}$ by way of describing the internal hom functor for the monoidal category $\Coll$, of collections and collection homomorphisms, with respect to the monoidal composition tensor product used to define globular operads, all the while emphasizing comparisons to the analogous construction in the category of graded sets.
\end{abstract}


\section{Introduction}\label{sec-Introduction}
The theory of operads has provided many insights and tools for modern mathematicians.  In particular, they provide a framework which can encode certain classes of algebraic theories.  Classical operads are limited, however, to algebraic theories in which the `operations' of the theory are of a single type.  If we wish to encode an algebraic theory which has operations which live in a higher-dimensional categorical structure, then a generalization of the classical notion of an operad is needed, namely the theory of globular operads.  These generalized operads have been studied extensively by both Leinster \cite{leinster2004higher} and Batanin \cite{batanin_1998GlobCat}.  The basic idea of a `globular' operad is to replace the set of operations in a classical operad with a globular set of operations.  Moreover, the numerical arity of an operation is generalized to the notion of an arity determined by a pasting scheme, all of whose cells are globes.  In this way, the process of composing operations in a classical operad by plugging $k$ operations into a single operation with $k$ inputs is generalized to the process of replacing $n$-cells in a globular pasting scheme with other $n$-dimensional pasting schemes to make a generalized composition on pasting operations of any finite dimension.  As Leinster and Batanin have shown, this generalization is particularly useful for keeping track of the infinitely many ways $n$-morphisms can be composed in a $\omega$-category.

The purpose of the present paper is not to explore more deeply or give further examples of such structures, but rather to provide insight into how one can think about algebras for globular operads.  Note that the results found in this exposition can be understood in greater generality using the theory of multicategories and $T$-spans (for a Cartesian monad $T$) as is done by Leinster in \cite{leinster2004higher}.  We intentionally avoid using the language of this theory here for the sake of making an introduction to globular operads that is more intuitive and accessible (and hopefully inspiring further exploration into the more abstract and technical details of the more general theory).  We shall do so by building intuition for globular operads and their algebras by comparison to the analogous construction for classical operads, thought of as monoids of graded sets.  Although further generalizations proceed more naturally at the more general level of multicategories and $T$-operads, it is the author's hope that describing the present examples in parallel may provide a helpful and accessible introduction to the constructions of the much larger and beautiful theory.

We will first review how to define algebras in the classical sense through the construction of the tautological operad on a set.  We shall then discuss how this construction can be alternatively described by thinking of operads as monoids in the category of graded sets.  Then, using some topos theoretic constructions, we give an explicit description of an internal hom in $\Grdset$ with respect to the tensor product used to define operads as monoids.  This allows us to define the tautological operad on a graded set so that we can explicitly define algebras for operads in this more general context.  All of this is done so that we can then explain, via comparison, how such a construction works when moving into the globular setting.  Note that both constructions follow the same procedure, in different contexts, on similar types of objects.  In particular, graded sets are replaced by collections and our operad composition is replaced with pasting compositions in all dimensions.  The final sections of this paper sketch this process using globular sets and collections. The construction of the tautological operad in the graded set setting is done in far greater detail so that the omitted details of the analogous construction for globular sets can be understood by analogy.

Before we begin, note that throughout this exposition we adopt the convention of using `classical operad' when we really mean traditional nonsymmetric operads.  We will not in this paper consider operads equipped with a symmetric group action.  However, we will often refer to these nonsymmetric operads simply as operads, with no other further adjectives attached.  The specific use of `classical' is intended only to distinguish between the globular and non-globular cases.

	\section{Basic definitions}\label{sec-basic-def}
	We begin by recalling the following standard definitions.  Further details can found in \cite{may2006geometry} or \cite{leinster2004higher}.

		\begin{definition}
			A \emph{nonsymmetric operad} $O$ consists of a sequence of sets $\{O(n)\}_{n \in \mathbb{N}}$ whose $n$-th entry is the set of $n$-ary operations, an identity operation $\id \in O(1)$, and for all $n,k_1,k_2,...,k_n \in \mathbb{N}$ a composition operation
			$$\circ: O(n) \times \displaystyle\prod_{i = 1}^{n}O(k_i) \rightarrow O(\displaystyle\sum_{i=1}^{n}k_i)$$
			such that
			$$\theta_0 \circ (\theta_1 \circ (\theta_{1_1}, ..., \theta_{1_k}), ..., \theta_n \circ (\theta_{n_1}, ..., \theta_{n_l})) = (\theta_0 \circ (\theta_1, ..., \theta_n)) \circ (\theta_{1_1},..., \theta_{1_k}, ...,\theta_{n_1}, ..., \theta_{n_l})$$
			and
			$$\theta_0 \circ (\id, \id, ..., \id) = \theta_0 = \id \circ \theta_0$$
			for all $\theta_i \in \{O(n)\}_{n \in \mathbb{N}}$ whenever the compositions are well-defined.
		\end{definition}
	
		\begin{definition}
			A \emph{homomorphism of operads} $f:O \rightarrow P$ is a sequence of maps \newline$\{f_n:O(n) \rightarrow P(n) \}_{n \in \mathbb{N}}$ that preserve both the identity operation and composition maps.
		\end{definition}
		
		\begin{definition}
			The \emph{tautological operad} $taut(X)$ on a set $X$ is the operad whose $n$-ary operations are $\Set(X^n,X)$, identity operation is the identity map $\id_X:X \rightarrow X$, and composition is given by
			$$\circ: \Set(X^n,X) \times \displaystyle\prod_{i = 1}^{n}\Set(X^{k_i},X) \rightarrow \Set(X^n,X) \times \Set(X^{\sum_{i}k_i},X^n) \rightarrow \Set(X^{\sum_{i}k_i},X)$$
			where first we take the disjoint union of maps in the second factor, keeping the first factor fixed, and then compose the resulting two factors as set maps.
		\end{definition}
			
		\begin{definition}
			An \emph{algebra} $A$ for an operad $O$ is a set $A$ equipped with an operad homomorphism $\xi:O \rightarrow taut(A)$.
		\end{definition}
				
		We now give a description of classical nonsymmetric operads in a different way which more naturally generalizes to the globular setting.  We begin with the following definition.
				
		\begin{definition}
			A \emph{graded set} is a set $X$ equipped with a function $x:X \rightarrow \mathbb{N}$ called the arity map.
		\end{definition}
					
		A graded set may also be thought of as a countably indexed family of sets in which for each $n \in \mathbb{N}$ the fiber $X_n := x^{-1}(n)$ is the set of `$n$-ary' elements.  In what follows, by abuse of notation we will often represent a graded set $x: X \rightarrow \mathbb{N}$ by its underlying set $X$.
					
		\begin{definition}
			Let
			$x: X \rightarrow \mathbb{N}$
			and
			$y: Y \rightarrow \mathbb{N}$
			be graded sets.  A \emph{morphism of graded sets} between them is a function $f:X \rightarrow Y$ which makes the following triangle commute:
			$$\xymatrix{X \ar[rr]^{f} \ar[dr]_{x} & & Y \ar[dl]^{y} \\
			& \mathbb{N} & & }$$
		\end{definition}
							
		Although the category of graded sets is the slice category $\Set / \mathbb{N}$, we shall denote it by $\Grdset$.  In what follows we will freely interchange the set $\mathbb{N}$ with $T(\{*\})$ where $(T:\Set \rightarrow \Set,\mu: T^2 \Rightarrow T, \eta: \id \Rightarrow T)$ is the free monoid monad on $\Set$.  This shift in perspective of thinking about $\mathbb{N}$ as the free monoid on the one point set $\{*\}$ will help us later to more naturally generalize this construction to the globular setting.  Moreover, we will shortly make use of the fact that the monad $T$ is cartesian $\cite{Leinster1998GenOp}$.
		
		\begin{definition}
			A monad $(T:\Set \rightarrow \Set,\mu: T^2 \Rightarrow T, \eta: \id \Rightarrow T)$ is a $\emph{cartesian monad}$ if all naturality squares for $\mu$ and $\eta$ are pullback squares and $T$ preserves all pullbacks.
		\end{definition}		
							
		The category $\Grdset$ has a second monoidal category structure different from both the cartesian and cocartesian structures.  We shall here denote this second monoidal product by $\mysquare$.
							
		\begin{definition}
			Let
			$x: X \rightarrow \mathbb{N}$
			and
			$y: Y \rightarrow \mathbb{N}$
			be a pair of graded sets.  Their \emph{composition tensor product}
			$x \mysquare y: X \mysquare Y \rightarrow \mathbb{N}$
			is defined by the diagram
			$$\xymatrix{X \mysquare Y \pullbackmark{0,2}{2,0} \ar[rr] \ar[dd] & & T(Y) \ar[r]^-{T(y)} \ar[dd]^{T(!_{Y})} & T^2(\{*\}) \ar[r]^-{\mu_{\{*\}}} & T(\{*\}) \\
			& &  \\
			X \ar[rr]^{x} & & T(\{*\})}$$
			where $!_Y:Y \rightarrow \{*\}$ is the unique map from $Y$ to the terminal one point set.  The underlying graded set $X \mysquare Y$ is the pullback of $x$ and $T(!_Y)$ with the arity function $x \mysquare y$ defined to be the composition along the top row.
		\end{definition}
									
		This definition makes $X \mysquare Y$ the graded set whose elements are pairs $(a,\psi)$ consisting of an element $a \in X$ and a word $\psi$ of elements from $Y$ with the property that the arity of $a$ agrees with length of $\psi$ given by $T(!_Y)$.  We then think of the elements of $X \mysquare Y$ as composable pairs consisting of elements from $Y$ that can be `plugged into' a single element from $X$.  Moreover, the arity of each pair is given as the sum of the arities of the entries in $\psi$ (as elements of $Y$).
		
		\begin{theorem}
			The product $\mysquare$ together with the terminal graded set $i: \{*\} \hookrightarrow \mathbb{N}$ gives $\Grdset$ the structure of a monoidal category.
		\end{theorem}
	
		\begin{proof}	
			The associator for $\Grdset$ with respect to $\mysquare$ can be obtained as follows.  Consider the graded sets $x:X \rightarrow T(\{*\})$, $y:Y \rightarrow T(\{*\})$, and $z:Z \rightarrow T(\{*\})$.  Then construct the graded sets
			$$x \mysquare y:X \mysquare Y \rightarrow T(\{*\})$$
			$$y \mysquare z:Y \mysquare Z \rightarrow T(\{*\})$$
			$$(x \mysquare y) \mysquare z:(X \mysquare Y) \mysquare Z \rightarrow T(\{*\})$$
			$$x \mysquare (y \mysquare z):X \mysquare (Y \mysquare Z) \rightarrow T(\{*\})$$
			as described above.  By definition, this means that the diagrams defining $(x \mysquare y) \mysquare z$ and $x \mysquare y$ can be configured together in the following way:
			$$\adjustbox{max width=\columnwidth}{\xymatrix{(X \mysquare Y) \mysquare Z \pullbackmark{0,2}{2,0} \ar[rrrrrr]^{\pi_2} \ar[dd]_{\pi_1} & &  & &  & & T(Z) \ar[rr]^{T(z)} \ar[dd]^{T(!_Z)} & & T^2(\{*\}) \ar[rr]^{\mu_{\{*\}}} & & T(\{*\}) \\
					\\
					X \mysquare Y \pullbackmark{0,2}{2,0} \ar[rr]^{\pi_2} \ar[dd]^{\pi_1} & & T(Y) \ar[rr]^{T(y)} \ar[dd]^{T(!_Y)} & & T^2(\{*\}) \ar[rr]^{\mu_{\{*\}}} & & T(\{*\}) \\
					\\
					X \ar[rr]^{x} & & T(\{*\})
				}}$$
			Note then that the top pullback square can be factored into three iterated pullback squares to obtain the following diagram:
			$$\adjustbox{max width=\columnwidth}{\xymatrix{(X \mysquare Y) \mysquare Z \pullbackmark{0,2}{2,0} \ar[rr]^{\phi} \ar[dd]_{\pi_1} & & T(Y \mysquare Z) \pullbackmark{0,2.5}{2.5,0} \ar[rr]^{T(\pi_2)} \ar[dd]^{T(\pi_1)} & & T^2(Z) \pullbackmark{0,2}{2,0} \ar[rr]^{\mu_{Z}} \ar[dd]^{T^2(!_Z)} & & T(Z) \ar[rr]^{T(z)} \ar[dd]^{T(!_Z)} & & T^2(\{*\}) \ar[rr]^{\mu_{\{*\}}} & & T(\{*\}) \\
					\\
					X \mysquare Y \pullbackmark{0,2}{2,0} \ar[rr]^{\pi_2} \ar[dd]^{\pi_1} & & T(Y) \ar[rr]^{T(y)} \ar[dd]^{T(!_Y)} & & T^2(\{*\}) \ar[rr]^{\mu_{\{*\}}} & & T(\{*\}) \\
					\\
					X \ar[rr]^{x} & & T(\{*\})
			}}$$
			Here the top right square is a pullback because it is a naturality square for $\mu$.  The top middle square is a pullback because it is the image of a pullback square under $T$, which preserves all pullbacks.  The top left square is then the pullback square which must exist by the fact that the single pullback we started with can be factored in this way by the right and middle pullbacks just described.  Now observe that the top left and bottom pullback squares together must form a pullback.  But then the pair of maps
			$$x \circ \pi_1: X \mysquare (Y \mysquare Z) \rightarrow T(\{*\})$$
			$$T(!_Y) \circ T(\pi_1) \circ \pi_2: X \mysquare (Y \mysquare Z) \rightarrow T(\{*\})$$
			give another pullback of the same cospan, inducing a map $\alpha_{X,Y,Z}: X \mysquare (Y \mysquare Z) \rightarrow (X \mysquare Y) \mysquare Z$ which we claim is the desired associator.  It remains to see that this map preserves the arity map for these two graded sets.  To see this, consider the following diagram:
			$$\xymatrix{T(Z) \ar[rr]^{T(z)}  & & T^2(\{*\}) \ar[rrdd]^{\mu_{\{*\}}} & & \\
				\\
				T^2(Z) \ar[uu]^{\mu_Z} \ar[rr]^{T^2(z)} & & T^3(\{*\}) \ar[uu]_{\mu_{T(\{*\})}} \ar[dd]^{T(\mu_{\{*\}})} & & T(\{*\}) \\
				\\
				T(Y \mysquare Z) \ar[uu]^{T(\pi_2)} \ar[rr]^{T(y \mysquare z)} & & T^2(\{*\}) \ar[rruu]_{\mu_{\{*\}}} & & \\
				}$$
			The bottom square is the definition of $T(y \mysquare z)$.  The top square commutes by the naturality of $\mu$.  The right square commutes by the associativity condition on $\mu$ as the multiplication transformation for $T$ as a monad.  The commutativity of the outer edges of this diagram then gives that
			$$\mu_{\{*\}} \circ T(y \mysquare z) = \mu_{\{*\}} \circ T(z) \circ \mu_Z \circ T(\pi_2)$$
			which together with the fact that
			$$\phi \circ \alpha_{X,Y,Z} = \pi_2: X \mysquare (Y \mysquare Z) \rightarrow T(Y \mysquare Z)$$
			shows that
			$$x \mysquare (y \mysquare z) = \mu_{\{*\}} \circ T(y \mysquare z) \circ \pi_2$$
			$$ = \mu_{\{*\}} \circ T(z) \circ \mu_Z \circ T(\pi_2) \circ \phi \circ \alpha_{X,Y,Z} = (x \mysquare y) \mysquare z \circ \alpha_{X,Y,Z}$$
			ensuring that the associator preserves arities, thus giving an isomorphism of graded sets.
			
			The monoidal identity for $\mysquare$ is the graded set $i: \{*\} \hookrightarrow \mathbb{N}$  where $i$ is simply the inclusion of the generator $*$ into the set $T(\{*\})$.  To see that this is the correct monoidal identity for $\mysquare$, notice that for any graded set $X$, $X \mysquare \{*\} = \{ (a,n)|a \in X, n \in T(\{*\}), x(a) = n \}$.  In other words, $X \mysquare \{*\}$ consists of pairs, an element from $X$ together with its arity.  This means that for each graded set $X \in \Grdset$ the $X$ component of the right unitor $\rho_X : X \mysquare \{*\} \rightarrow X$ is simply first projection with its inverse given by the graded set inclusion map $\overline{\rho_X}: X \hookrightarrow X \mysquare \{*\}$ which couples each element in $X$ with its arity.  By swapping the two variables we get that $\{*\} \mysquare X = \{ (*,\psi)|\psi \in T(X), T(!_X)(\psi) = * \}$.  Hence $\{*\} \mysquare X$ consists of pairs, the singleton $\{*\}$ and a word of length one from $T(X)$.  But words of length one in $T(X)$ are exactly the elements of $X$.  This then implies that the $X$ component of the left unitor $\lambda_X : \{*\} \mysquare X \rightarrow X$ must be second projection with inverse given by the graded set inclusion map $\overline{\lambda_X}: X \hookrightarrow \{*\} \mysquare X$ which couples each element in $X$ with its arity, but on the opposite side as that of $\overline{\rho_X}$.
			
			It then remains only to show that the pentagon coherence condition follows.  But this is clear from the fact that each component of the associator follows from a universal construction.  The triangle identities follow immediately from the fact that each component of the left and right unitors is simply a projection map.
		\end{proof}

		\begin{theorem}
			A nonsymmetric operad is a monoid in $\Grdset$ with respect to the monoidal product $\mysquare$.
		\end{theorem}
										
		\begin{proof}
			A monoid in $\Grdset$ consists of an underlying graded set $x:X \rightarrow \mathbb{N}$, thought of as a set of `operations', together with a composition function $m : X \mysquare X \rightarrow X$ and a unit function $e : \{*\} \rightarrow X$ from $\Grdset$, all of which must satisfy the usual associativity and unital conditions.  We denote $n$-ary operations of the underlying set $X$ by $X_n$, which is simply the fiber over $n$ along the arity map $x$.  These fibers then form the sequence of sets $\{X_n\}_{n \in \mathbb{N}}$ for a nonsymmetric operad.  The function $m$ keeps track of how to compose strings of elements in $X$ with an element $a \in X$ of the appropriate arity. The function $e$ distinguishes an element of $X$ which will behave like an identity operation on the elements of $X$.  The associativity and unit commutative diagrams for a monoid internal to a category then ensure that these graded set maps endow $X$ with the needed associative and unital operadic composition with respect to the $\mysquare$ product.  Conversely, given a nonsymmetric operad $O$, its underlying graded set can be obtained by constructing a map whose fiber over $n$ is exactly the $n$th term in the sequence $\{O(n)\}_{n \in \mathbb{N}}$.  By construction, each well-defined composition in $O$ can be identified with an element of $O \mysquare O$.  Hence the composition map $m: O \mysquare O \rightarrow O$ is defined to be sending each composable pair to their composite in $O$.  As $O$ has a distinguished identity element, the map $e:\{*\} \rightarrow O$ sends * to this distinguished element.  The associativity and unital conditions required of $O$ then guarantee that the needed commutative diagram conditions are satisfied.
		\end{proof}

	\section{Exponentials and the internal hom in \Grdset}\label{sec-exp-grdset}
		In this section we will recall several topos theoretic constructions in order to better understand the category $\Grdset$.  Our primary goal will be to understand the right adjoint to the functor $- \mysquare B:\Grdset \rightarrow \Grdset$ which, for a fixed graded set $B$, sends a graded set $A$ to the product graded set $A \mysquare B$.  Letting the graded set $B$ be variable then allows us to compute the internal hom in $\Grdset$ with respect to the product $\mysquare$.  We will first review how three canonical functors can be formed from a single morphism of graded sets (further details can be found in \cite{johnstone2002sketches}, \cite{johnstone2014topos}, and \cite{maclane1994sheaves}).  We will then, as a warm up, see how these functors can be used to construct a right adjoint to the Cartesian product in $\Grdset$, allowing us to compute exponential objects $A^B$ between graded sets.  We will then show how this procedure can be slightly modified to create the desired right adjoint to $- \mysquare B$.
  
        Let $f: A \rightarrow B$ be a set map.  There is then an induced functor $f^*:\Set/B \rightarrow \Set/A$ between slice categories called a $\textit{change of base}$ functor.  It takes a set map $\chi : X \rightarrow B$ and returns the pullback map $f^{*}(\chi):X \pullbackSub{\chi}{f} A \rightarrow A$ of $\chi$ along $f$.  It is furthermore known that for each $f$ the functor $f^*$ has both a left and right adjoint.  On objects its left adjoint $\Sigma_f: \Set/A \rightarrow \Set/B$ is simply composing an object of $\Set/A$ with $f$ resulting in an object in $\Set/B$.  The right adjoint $\Pi_f: \Set/A \rightarrow \Set/B$ is however a bit more complicated.  Nonetheless, when our base category is $\Set$, it has a fairly straight forward description as follows.  Let $\psi : Y \rightarrow A$ be any morphism in $\Set/A$.  The map $\Pi_f(\psi): \Gamma \rightarrow B$ is constructed by specifying the fiber over each point as follows.  Take an element $b \in B$ and consider its fiber $A_b$ along the map $f$.  Each element $c \in A_b$ then has a fiber $Y_c$ sitting above it along the map $\psi$.  We can then define the fiber $\Gamma_b$ along the map $\Pi_f(\psi)$ to be the product $\displaystyle\prod_{c \in A_b} Y_c$.  Following this construction for each $b \in B$ gives the complete map from $\Gamma := \displaystyle\coprod_{b \in B} \displaystyle\prod_{c \in A_b} Y_c$ to $B$.
										
		Let us now look at how these maps can be used to construct exponential objects in $\Grdset$.  Consider the functor $- \times B: \Grdset \rightarrow \Grdset$ for a fixed graded set $b:B \rightarrow \mathbb{N}$.  It can be written as a composition of the functors defined above.  We get that
		$$- \times B = \Sigma_{b}b^*$$
		since cartesian product in a slice category is given by the pullback of the two factors and $\Grdset$ is itself a slice category.  Writing the functor $- \times B$ in this way allows us to immediately compute its right adjoint $-^{B}:\Grdset \rightarrow \Grdset$, which is the exponentiation by $B$ functor.  This is done by taking the right adjoint of each factor in the composition and reversing the order in which they are composed, which leads to the following formula:
		$$-^B = \Pi_{b}b^*$$
		Note that $b^*(a): A \pullbackSub{a}{b} B \rightarrow B$ is simply second projection.  Hence, this functor takes a graded set $a:A \rightarrow \mathbb{N}$ and applies the fiber-wise construction for $\Pi_{b}(b^{*}(a))$ described above to get, for each $n \in \mathbb{N}$,
		$$\Gamma_n = \displaystyle\prod_{y \in B_n}\{(x,y)|x \in A, a(x)=b(y)\} \cong \displaystyle\prod_{y \in B_n}\{x|x \in A, a(x)=n\}$$
		is the fiber over $n$ along $A^B: \Gamma \rightarrow \mathbb{N}$.  This allows us to think of the elements in each fiber $\Gamma_n$ as a choice of how to associate to each element of arity $n$ from $B$ an element of arity $n$ in $A$.  In other words, it defines a map from $B_n$ to $A_n$.  Recall though that all maps of graded sets carry $n$ fibers to $n$ fibers by definition.  Thus we can think of the exponential object $A^B$ as the set of maps from $B$ to $A$ `cut up' into their $n$th fiber restrictions for each $n \in \mathbb{N}$.
										
		Now consider the functor $- \mysquare B: \Grdset \rightarrow \Grdset$ for the graded set $b:B \rightarrow \mathbb{N}$.  We will construct the internal hom with respect to $\mysquare$ using a similar construction to that of the exponential object above.  We can write $- \mysquare B$ as the following composition:
		$$- \mysquare B = \Sigma_{\mu_{\{*\}}}\Sigma_{T(b)}T(!_B)^*$$
		Notice that this functor takes the graded set $a:A \rightarrow \mathbb{N}$ to the graded set $a \mysquare b: A \mysquare B \rightarrow \mathbb{N}$, where the arity map $a \mysquare b$ is exactly the image of $\Sigma_{\mu_{\{*\}}}\Sigma_{T(b)}T(!_B)^*(a)$.  We see in the diagram below that this is exactly the topmost horizontal composition in the diagram used earlier to define the composition tensor product $\mysquare$ in $\Grdset$.
		$$\xymatrix{A \mysquare B \pullbackmark{0,2}{2,0} \ar[rr]^-{T(!_B)^*} \ar[dd] & & T(B) \ar[r]^-{T(b)} \ar[dd]^{T(!_B)} & T(\mathbb{N}) \ar[r]^-{\mu_{\{*\}}} & \mathbb{N} \\
		& &  \\
		A \ar[rr]^{a} & & \mathbb{N}}$$
		Writing the functor $- \mysquare B$ in this way, just as with $- \times B$ above, allows us to immediately compute its right adjoint $[B,-]:\Grdset \rightarrow \Grdset$ by taking the right adjoint of each factor in the composition and reversing the order in which they are composed.  This then leads to the following formula.
		$$[B,-] = \Pi_{T(!_B)}T(b)^*\mu_{\{*\}}^*$$
		We shall first consider how the composite $T(b)^*\mu_{\{*\}}^*$ acts on a graded set $a:A \rightarrow \mathbb{N}$.  Recall that the map $T(b)^*\mu_{\{*\}}^*(a)$ is given as the topmost edge in the following double pullback diagram.
		$$\xymatrix{(A \pullbackSub{a}{\mu_{\{*\}}^{*}} T(\mathbb{N})) \pullbackSub{\mu_{\{*\}}^{*}(a)}{T(b)} T(B) \pullbackmark{0,2.75}{2.75,0} \ar[rr]^-{T(b)^*\mu_{\{*\}}^*(a)} \ar[dd]_{\pi_1} & & T(B) \ar[dd]^{T(b)} \\
		\\
		A \pullbackSub{a}{\mu_{\{*\}}^{*}} T(\mathbb{N}) \pullbackmark{0,2.75}{2.75,0} \ar[rr]^-{\mu_{\{*\}}^*(a)} \ar[dd]_{\pi_1} & & T(\mathbb{N}) \ar[dd]^{\mu_{\{*\}}} \\
		\\
		A \ar[rr]^a & & \mathbb{N} \\
		}$$
		More concretely, for every $\beta \in T(B)$ there is a fiber over it along $T(b)^*\mu_{\{*\}}^*(a)$ living in the set $A \pullbackSub{a}{\mu_{\{*\}}^{*}} T(\mathbb{N})$ consisting of pairs $(\alpha, t)$ with $\alpha \in A$ and $t$ a word of natural numbers such that the arity of $\alpha$ is the sum of the arities in each slot of the `operation' $t$.  Moreover, the letters of $t$ give the arities of the letters in $\beta$ respectively.  We now apply $\Pi_{T(!_B)}$ to $T(b)^*\mu_{\{*\}}^*(a)$ to get the internal hom in $\Grdset$.  Recall from above that the map $\Pi_{T(!_B)}T(b)^*\mu_{\{*\}}^*(a): \Gamma \rightarrow \mathbb{N}$ is constructed by specifying the fiber over each point.  So take any $n \in \mathbb{N}$ and consider its fiber $T(B)_n$ along the map $T(!_B):T(B) \rightarrow \mathbb{N}$.  Each element $\beta \in T(B)_n$ then has a fiber $((A \pullbackSub{a}{\mu_{\{*\}}^{*}} T(\mathbb{N})) \pullbackSub{\mu_{\{*\}}^{*}(a)}{T(b)} T(B))_{\beta}$ sitting above it along the map $T(b)^*\mu_{\{*\}}^*(a)$.  We can then define the fiber $\Gamma_n$ along the map $\Pi_T(!_B)T(b)^*\mu_{\{*\}}^*(a)$ to be the following product:
		$$\displaystyle\prod_{\beta \in T(B)_n} ((A \pullbackSub{a}{\mu_{\{*\}}^{*}} T(\mathbb{N})) \pullbackSub{\mu_{\{*\}}^{*}(a)}{T(b)} T(B))_{\beta}$$
		Following this construction for each $n \in \mathbb{N}$ gives the complete map.  Thus $\Grdset$ is right closed with respect to the composition tensor product.
												
		Via this construction, we can now compute the internal hom $H_{B,A}: [B,A] \rightarrow \mathbb{N}$ in $\Grdset$ between any two graded sets $b:B \rightarrow \mathbb{N}$ and $a:A \rightarrow \mathbb{N}$.  We can think of the $n$-ary elements of the underlying set $[B,A]$ as follows.  An element $\beta \in T(B)_n$ can be thought of as a choice of $n$ elements $\{\beta_1,\beta_2,...,\beta_n\}$ from $B$ juxtaposed together such that $\beta = \beta_1\beta_2...\beta_n$.  This allows us to think of the fiber
		$$[B,A]_n = \displaystyle\prod_{\beta \in T(B)_n}\{((p,w), \beta)|p \in A, w \in T(\mathbb{N})_n, a(p) = \sum_{i}w_i, T(b)(\beta)=w\}$$
		instead as the set
		$$[B,A]_n \cong \displaystyle\prod_{\beta \in T(B)_n}\{(p,\beta)|p \in A, a(p) = \sum_{i}b(\beta_i)\}$$
		up to isomorphism.  Hence, an element $\gamma \in [B,A]_n$ may be thought of as a choice of elements from $A$ to correspond to each string of $n$ elements from $B$ in such a way as to preserve arities, making $\gamma$ a map of graded sets.  In other words, a `map' in the internal hom is essentially a thing that takes $n$ elements from the source and picks an element of the target that has arity equal to the sum of the arities of the $n$ elements from the source.

	\section{The classical tautological operad}\label{sec-class-taut}
		Consider a graded set $x:X \rightarrow \mathbb{N}$.  We shall now construct the tautological operad on $X$, denoted again by $taut(X)$.  First consider the graded set $taut(X):= [X,X]$ obtained using the internal hom in $\Grdset$.  The underlying graded set for the tautological operad on $X$ can be thought of as abstractly encoding all the possible operations that take some number of elements from $X$ to a single output from $X$.
												
		\begin{theorem}
			The graded set $taut(X)$ admits the structure of a nonsymmetric operad.
		\end{theorem}
		\begin{proof}
		We must show that $taut(X)$ is a monoid in $\Grdset$ with respect to the tensor product $\mysquare$.  The operad identity is given by the map $e: \{*\} \rightarrow [X,X]$ which is constructed as the currying of the left unitor $\lambda_X: \{*\} \mysquare X \rightarrow X$ for the monoidal structure in $\Grdset$.  The composition map $m:[X,X] \mysquare [X,X] \rightarrow [X,X]$ is similarly constructed in the following way.  Consider the counit $\epsilon_A: [A,-] \mysquare A \Rightarrow \id_{\Grdset}$ of the hom-tensor adjunction in $\Grdset$ between $- \mysquare A$ and $[A,-]$, which we can think of as $\textit{evaluation at}$ $A$.  It has components
		$\epsilon^B_A:[A,B]\mysquare A \rightarrow B$ for each graded set $A$.  We then get a map
		$$\kappa: [X,X] \mysquare ([X,X] \mysquare X) \rightarrow [X,X] \mysquare X \rightarrow X$$
		which is the composite $\kappa := \epsilon^X_X(\id_X \mysquare \epsilon^X_X)$.  The operad multiplication for $[X,X]$ is then the currying of the map $\kappa$.  It then remains only to show that $taut(X):[X,X] \rightarrow \mathbb{N}$ together with $e: \{*\} \rightarrow [X,X]$ and $m: [X,X] \mysquare [X,X] \rightarrow [X,X]$ satisfy the commutative diagrams required of a monoid object in $\Grdset$.  This can be seen by first currying the maps in the relevant diagrams and checking to see that these new curred diagrams, whose commutativity is equivalent with that of the originals, commute.  Here we have dropped all parentheses by MacLane's coherence theorem applied to $\Grdset$.  We first consider the diagram
		$$\adjustbox{max width=\columnwidth}{\xymatrix{[X,X] \mysquare [X,X] \mysquare [X,X] \mysquare X \ar[rr]^-{m \mysquare \id_{[X,X]} \mysquare \id_X} \ar[rrdd]^{\hskip2em \id_{[X,X]} \mysquare \id_{[X,X]} \mysquare \epsilon_X^X} \ar[dddd]_{\id_{[X,X]} \mysquare m \mysquare \id_{X}} & & [X,X] \mysquare [X,X] \mysquare X \ar[rr]^-{\id_{[X,X]} \mysquare \epsilon_X^X} & & [X,X] \mysquare X \ar[dddd]^{\epsilon_X^X} \\
		\\
		& & [X,X] \mysquare [X,X] \mysquare X \ar[rruu]^{m \mysquare \id_X} \ar[dd]^{\id_{[X,X]} \mysquare \epsilon_X^X} & & \\
		\\
		[X,X] \mysquare [X,X] \mysquare X \ar[rr]^-{\id_{[X,X]} \mysquare \epsilon_X^X} & & [X,X] \mysquare X \ar[rr]^{\epsilon_X^X} & & X
		}}$$
		whose commutativity is equivalent with that of the diagram asserting associativity of our multiplication $m$.  Note that the commutativity of the bottom left and right squares follows from the way we defined $m$ as the currying of two sequential multiplication operations.  The remaining top square then commutes by the functoriality of the composition tensor product.  We next consider the diagram
		$$\adjustbox{max width=\columnwidth}{\xymatrix{\{*\} \mysquare [X,X] \mysquare X \ar[rr]^-{e \mysquare \id_{[X,X]} \mysquare \id_X} \ar[dd]_{\id_{\{*\}} \mysquare \epsilon^X_X} & & [X,X] \mysquare [X,X] \mysquare X \ar[rr]^-{\id_{[X,X]} \mysquare \epsilon^X_X} & & [X,X]\mysquare X \ar[dd]^{\epsilon^X_X} \\
		\\
		\{*\} \mysquare X \ar[rrrruu]^{e \mysquare \id_{X}} \ar[rrrr]_{\lambda_{X}} & & & & X
		}}$$
		which comes from currying the maps from the needed left-sided unit diagram.  Here the top left square commutes by the functoriality of the composition tensor product.  The bottom right triangle commutes by the definition of $e$.  We finally consider the diagram
		$$\adjustbox{max width=\columnwidth}{\xymatrix{[X,X] \mysquare \{*\} \mysquare X \ar[rr]^-{\id_{[X,X]} \mysquare e \mysquare \id_X} \ar[dd]_{\id_{[X,X]} \mysquare \lambda_X} & & [X,X] \mysquare [X,X] \mysquare X \ar[rr]^-{\id_{[X,X]} \mysquare \epsilon^X_X} \ar[ddll]^{\id_{[X,X]} \mysquare \epsilon^X_X} & & [X,X] \mysquare X \ar[dd]^{\epsilon^X_X} \\
		\\
		[X,X] \mysquare X \ar[rrrr]_{\epsilon^X_X} & & & & X
		}}$$
		which comes from currying the maps from the needed right-sided unit diagram.  The top left triangle of this diagram commutes by the definition of $e$.  The bottom right square commutes trivially.  It therefore follows that $(taut(X),m,e)$ is a monoid in $\Grdset$.
		\end{proof}
																				
		Note that the operad identity map $e: \{*\} \rightarrow [X,X]$ is the canonical morphism that maps the singleton $*$ to the element of [X,X] corresponding to the identity set map on $X$, while the composition map $m:[X,X] \mysquare [X,X] \rightarrow [X,X]$ is the canonical map which takes a pair $(a,\omega) \in [X,X] \mysquare [X,X]$ and composes each of the letters from the word $\omega \in T([X,X])$ with each of the respective inputs for the `operation' $a \in [X,X]$.  Hence we always have a natural way to equip $taut(X)$ with the structure of a nonsymmetric operad in \Grdset.  This construction allows us to make the following definitions.

        \begin{definition}
            A graded set $X$ is called a $\emph{degenerate graded set}$ if the arity map factors as $x = [0] \circ !_{X}$, where $[0]: \{*\} \rightarrow T(\{*\})$ is the `name of zero' map which identifies the empty word in $T(\{*\})$.
        \end{definition}

        We say that a degenerate graded set is concentrated over zero, or in degree zero, as they are precisely those graded sets with the property that all elements have arity zero.  Note that every set $X$ can canonically be associated to a degenerate graded set by equipping $X$ with the arity map that sends every element to zero.  This allows us to concisely define algebras for an operad.
																				
		\begin{definition}
			Let $o:O \rightarrow \mathbb{N}$ be an operad.  A $\emph{left }O\emph{-module}$ is a morphism of operads $f:O \rightarrow taut(X)$ for some graded set $x:X \rightarrow \mathbb{N}$.  An $O\emph{-algebra in}$ $\Set$ is a left $O$-module such that $X$ is a degenerate graded set. 
		\end{definition}
																					
		In the situation above, we say that the operad $O$ acts on the graded set $X$.  Note here that the algebras for an operad in this sense would be more general than that of the algebras for a classically defined operad if we did not require that $X$ be concentrated over zero.  These graded algebras are referred to as left modules \cite{Fresse2009Modules}. In order for the operad to act on the set $X$ as an ordinary set, as opposed to a graded set, we must think of ordinary sets as being graded sets concentrated over zero.  In this way the $n$-th component of the tautological operad defined on the graded set $x:X \rightarrow \mathbb{N}$ concentrated over zero corresponds exactly to the classic definition of the tautological operad on the set $X$ in the sense that $[X,X]_n$ would consist of pairings of a length $n$ word from $X$, whose letters each have arity zero (i.e. $n$ elements from the set $X$), to an element of $X$ whose arity is the sum of the arities of the $n$ letters which comprised the source word, which is also zero.  Hence $[X,X]_n$ can be thought of as associating to each string of $n$ elements from $X$ a single element of $X$.  Thus $[X,X]_n \cong \Set(X^n,X)$ in $\Set$.

        In \cite{leinster2004higher} Leinster gives several equivalent definitions of an algebra for an operad.  One key definition notes that the free monoid monad $T$ can be used to construct a monad $T_O$ associated to a fixed operad $O$.  The monad sends a set $X$ to the set $T_O(X)=\coprod_{n \in \mathbb{N}} O(n) \times X^n$, with the multiplication and unit structure for the operad $O$ inducing a monad structure on $T_O$.  Then, an algebra for the monad $T_O$ is exactly the same as an algebra for the operad $O$.  More concretely, an algebra for the operad $O$ is precisely a set $X$ together with a structure map $h:T_O(X) \rightarrow X$ satisfying the necessary compatibility with composition and identities in $O$.

        We can rephrase this notion of algebra using the structure of $\mysquare$ in the following way.  Remember that any set $X$ can be considered a degenerate graded set by equiping it with the arity map that sends every element to $0$.  So we can then associate the set $T_O(X)$ with the graded set $O \mysquare X$.  The algebra structure map $h:T_O(X) \rightarrow X$ can be thought of as a graded set map $h:O \mysquare X \rightarrow X$ which can be curried using the closed structure for $\mysquare$ to get a graded set map $O \rightarrow [X,X]$ which is precisely the notion of an algebra given above.
																					
		We will see a similar relationship between these notions of algebra later when discussing the analogous construction for globular set.  Before moving on, we immediately get the following theorem by defining algebras this way.
																					
		\begin{theorem}
			An algebra for an operad $O$ is an algebra for every operad $P$ which maps to $O$.  In particular, an algebra for $O$ is an algebra for every sub-operad of $O$.
		\end{theorem}
																						
		\begin{proof}
			Recall that any operad $O$ is a monoid in $\Grdset$ and an algebra for such an operad is specified by a morphism of graded sets $f: O \rightarrow taut(X)$.  Let $P$ be another operad and $g: P \rightarrow O$ be a morphism of graded sets.  Then the map $f(g): P \rightarrow taut(X)$ induces on $X$ the structure of a $P$-algebra.  Moreover, as any sub-operad $\tilde{O}$ of $O$ comes from removing a certain subset of elements from $O$, the algebra for $O$ specified by the map $f$ is also an algebra for $\tilde{O}$ simply by restricting $f$ to $\tilde{O}$.  In fact, any further subsets of $\tilde{O}$ corresponds to a further restriction of $f$, showing that any algebra for $O$ is an algebra for every sub-operad $\tilde{O}$. 
		\end{proof}

	\section{Collections}\label{sec-collections}
		We begin this section by recalling the notion of a globular set.  To do so requires the following category $\mathbb{G}$, known as the \textit{globe category}.  The category $\mathbb{G}$ has $\mathbb{N}$ as its set of objects.  Its morphisms are generated by $\sigma_n: n \rightarrow n+1$ and $\tau_n: n \rightarrow n+1$ for all $n \in \mathbb{N}$ subject to the relations $\sigma_{n+1} \circ \sigma_n = \tau_{n+1} \circ \sigma_n$ and $\sigma_{n+1} \circ \tau_n = \tau_{n+1} \circ \tau_n$.

		\begin{definition}
			A \emph{globular set} is a contravariant functor $G: \mathbb{G} \rightarrow \Set$. The category $\Glob$ of globular sets is the category of presheaves on $\mathbb{G}$.
		\end{definition}

		More concretely, a globular set $G = (\{G_n\}_{n\in \mathbb{N}},\{s_G^n\},\{t_G^n\})$ is specified by the following set of data: a countable collection of sets $\{G_n\}_{n \in \mathbb{N}}$, where each $G_n$ is called the set of $n$-cells of $G$, together with source and target maps $s_G = \{s_G^n:G_n \rightarrow G_{n-1}\}$ and $t_G = \{t_G^n:G_n \rightarrow G_{n-1}\}$ subject to the relations $s_G^n \circ s_G^{n+1} = s_G^n \circ t_G^{n+1}$ and $t_G^n \circ s_G^{n+1} = t_G^n \circ t_G^{n+1}$ in each dimension $n \in \mathbb{N}$.
		
		\begin{definition}
			Let $G: \mathbb{G} \rightarrow \Set$ and $H: \mathbb{G} \rightarrow \Set$ be globular sets. A \emph{globular set homomorphism} is a natural transformation $\varphi: G \Rightarrow H$ between globular sets.  In particular, $\varphi = \{\varphi_n:G_n \rightarrow H_n\}_{n \in \mathbb{N}}$ is given by a sequence of set maps which make the two diagrams
			$$\xymatrix{G_n \ar[rr]^{\varphi_n} \ar[dd]_{s_G^n} & & H_n \ar[dd]^{s_H^n} & & G_n \ar[rr]^{\varphi_n} \ar[dd]_{t_G^n} & & H_n \ar[dd]^{t_H^n} \\
						\\
						G_{n-1} \ar[rr]^{\varphi_{n-1}} & & H_{n-1} & & G_{n-1} \ar[rr]^{\varphi_{n-1}} & & H_{n-1} \\
			}$$
			commute for all $n \in \mathbb{N}$.
		\end{definition}
		
		Together with the natural transformations between them, globular sets form a category which we shall here denote $\Glob$.

		Another integral piece of structure needed to define globular operads is the free strict $\omega$-category monad $\mathcal{T}: \Glob \rightarrow \Glob$.  Just like the monad $T$ above, $\mathcal{T}$ is cartesian.  This fact, as well as a more detailed explanation of its construction and use, can be found in \cite{leinster2004higher}.  Briefly, this monad takes a globular set $\mathcal{X}$ and returns the underlying globular set of the free strict $\omega$-category generated by $\mathcal{X}$.  In other words, it takes a globular set $\mathcal{X}$ and constructs the globular set $\mathcal{T}(X)$ consisting of all possible pasting diagrams, or as we will often describe them, `globular words' built out of the cells of $\mathcal{X}$.  The motivation for calling such a pasting diagram a word is that a pasting diagram, all of whose cells are cells in $\mathcal{X}$, can be thought of as a generalization of the notion of a word in some set $Y$ (i.e. a string of concatenated elements from $Y$).  The main difference between the two notions is that a globular word can be built out of concatenation of cells along any of their boundary cells, as opposed to the classical setting in which elements, or letters, can only be composed as horizontal strings.  So in this way we can think of words on a set (in either setting) as simply an element, or cell, in the underlying object of the free monoid, or $\omega$-category, on the respective notion of set.

		For our purposes we will be specifically interested in the globular set $\Tone$ generated by the terminal globular set $\one$ which has exactly one cell in each dimension.  It is precisely $\Tone$ which allows us to generalize our notion of the arity of an operation.  In the classical case, arity is essentially the `word length' of the element over which the operation sits with respect to the operad's underlying graded set structure.  In this more general context, the arity of a cell in a globular set $\mathcal{X}$ is the `shape', or more precisely the globular pasting diagram which names a cell in $\Tone$, specified by the globular set map from $\mathcal{X}$ to $\Tone$.

		\begin{definition}
			A \emph{collection} is a globular set $\mathcal{X}$ equipped with a globular set homomorphism $x:\mathcal{X} \rightarrow \Tone$ called the \emph{arity map}.
		\end{definition}

		It is often convenient to use the diagram
		$$\xymatrix{\mathcal{X} \ar[d]_{x} \\
		\Tone}$$
		to represent a collection in order to emphasize that the cells in $\mathcal{X}$ `sit over' a specified `shape' in $\Tone$.  However, for readability, just as with graded sets, we will often represent a collection by simply writing its underlying globular set $\mathcal{X}$.

		\begin{definition}
			Let $x: \mathcal{X} \rightarrow \Tone$ and $y: \mathcal{Y} \rightarrow \Tone$ be a pair of collections.  A \emph{collection homomorphism} between them is a globular set map $f:\mathcal{X} \rightarrow \mathcal{Y}$ which makes the triangle
			$$\xymatrix{\mathcal{X} \ar[rr]^{f} \ar[dr]_{x} & & \mathcal{Y} \ar[dl]^{y}\\
			& \Tone & }$$
			commute.
		\end{definition}

		We shall use $\Coll$ to denote the category of collections. Note that $\Coll$ is simply the slice category $\Glob / \Tone$.  Furthermore, $\Coll$, as with $\Grdset$, has a monoidal structure with respect to the composition tensor product $\mysquare : \Coll \times \Coll \rightarrow \Coll$ defined analogously as follows:
		
		\begin{definition}
			Let $x: \mathcal{X} \rightarrow \Tone$ and $y: \mathcal{Y} \rightarrow \Tone$ be a pair of collections.  Their composition tensor product $x \mysquare y : \mathcal{X} \mysquare \mathcal{Y} \rightarrow \Tone$ is defined by the diagram:
			$$\xymatrix{\mathcal{X} \mysquare \mathcal{Y} \pullbackmark{0,2}{2,0} \ar[rr] \ar[dd] & & \mathcal{T}(\mathcal{Y}) \ar[r]^-{\mathcal{T}(y)} \ar[dd]^{\mathcal{T}(!_{\mathcal{Y}})} & \Tsquared \ar[r]^-{\mu_{\one}} & \Tone \\
				& &  \\
				\mathcal{X} \ar[rr]^{x} & & \Tone}$$
			where $!_{\mathcal{Y}}: \mathcal{Y} \rightarrow \one$ is the unique map from $\mathcal{Y}$ to the terminal globular set.  The underlying globular set $\mathcal{X} \mysquare \mathcal{Y}$ is the pullback of $x$ and $\mathcal{T}(!_{\mathcal{Y}})$ with the arity globular set map $x \mysquare y$ defined to be the composition along the top row.
		\end{definition}
		
		This definition makes $\mathcal{X} \mysquare \mathcal{Y}$ the unique collection whose cells are pairs $(a,\psi)$ consisting of a cell $a \in \mathcal{X}$ and a `globular word' $\psi$ of cells from $\mathcal{Y}$ indexed by the arity shape of $a$.  In $\mathcal{X} \mysquare \mathcal{Y}$, the `globular letters' in the globular word $\psi \in \mathcal{T}(\mathcal{Y})$ may be compatibly `glued together' along the shape of $x(a) \in \Tone$ in the sense that each globular letter is a $k$-cell that can replace a particular $k$-cell in the pasting diagram $x(a)$. We can thus think of the cells of $\mathcal{X} \mysquare \mathcal{Y}$ as composable pairs specified by a cell of $\mathcal{X}$ and `words of cells' over $\mathcal{Y}$ indexed by the shape of the cell from $\mathcal{X}$.  Furthermore, the arity for cells in $\mathcal{X} \mysquare \mathcal{Y}$ comes as an extension of the arity for cells of $\mathcal{Y}$ when they have been `plugged into' the cell $a$, or rather, `glued together' in the `shape' of $a$ specified by the morphism $x$.  To see this, note that the map $\mathcal{T}(y)$ takes a word of cells from $Y$ and returns a shape of shapes (i.e. a cell in $\mathcal{T}^2(\one)$ which is named by a pasting diagram of pasting diagrams).  The multiplication transformation $\mu_{\one}$ at $\one$ then takes this shape of shapes and returns the shape we would get if we replaced the cells in $x(a)$ with the specified arity shapes of each letter from $\mathcal{Y}$.  More precisely, $\mu_{\one}$ performs the pasting compositions which sends a pasting diagram of pasting diagrams to the particular pasting diagram obtained by replacing the $k$-cells of the `outer diagram' with the $k$-dimensional pasting diagrams, or the `inner diagrams', indexed by the cells of outer diagram.  It is analogous to the operation of taking a word of words from the free monoid on the free monoid on a set $X$ and sending it to a single word in the free monoid on $X$ by concatenating the inner words as letters of the outer word.  In a sense, the inner words are `plugged into' the $k$th letter slot of the outer word to make a composite word.  In the globular setting however, the plugging in occurs in the same fashion, but with the letter slots in each word being of varying dimensions concatenated along their various boundary cells, rather than simple horizontal strings of letters from a set.

		Before moving on, there are several special collections that are worth noting.  The first is the terminal collection $\Tone \rightarrow \Tone$ which functions as the unit for the Cartesian product in $\Coll$.  There is also the initial trivial collection $\emptyset \rightarrow \Tone$ whose arity map is the vacuous mapping from the empty globular set.  There is the collection $I: \one \hookrightarrow \Tone$ whose arity map is simply the inclusion of generators.  This collection is important because it is the unit for $\mysquare$ in $\Coll$ just as $i:\{*\} \hookrightarrow T(\{*\})$ was the unit for $\mysquare$ in $\Grdset$.  It is worth noting that this inclusion of generators into $\Tone$ is simply the component at $\one$ of the unit map $\id \Rightarrow \mathcal{T}$ for the monad $\mathcal{T}$.
		
		One final and important collection is given by the globular set map $[id]: \one \rightarrow \Tone$.  To understand this collection, note first that among the many cells in $\Tone$ are the underlying globular cells of identity morphisms created when $\mathcal{T}$ produces the underlying globular set of the free strict $\omega$-category on $\one$.  Among these identities are the following special identities.  There is the underlying 1-cell of the identity on the single vertex in $\one$.  This identity map then has an identity 2-cell that sits over it.  And over this identity 2-cell there is an identity 3-cell that sits above it.  Continuing this process, we see that there is an inclusion of the terminal object $\one$ into $\Tone$ whose cells are exactly the iterated identities of the single object.  This sub-object can be thought of as the globular $\omega$-analogue of the additive identity $0 \in \mathbb{N}$ from the graded set case.  The map $[id]$ is then the globular set map which sends each of the single $n$-cells in $\one$ to the corresponding iterative identity cell of dimension $n$ from the construction just described.  The map $[id]$ is, in this way, an identification of this `tower' of iterated identities as the particular identities for each $n$-dimensional pasting composition.

	\section{Globular operads}\label{sec-glob-operads}
		We can now see that $\Coll$ has the following monoidal structure.
		
		\begin{theorem}
			The product $\mysquare$ together with the collection $I: \one \hookrightarrow \Tone$ gives $\Coll$ the structure of a monoidal category.
		\end{theorem}
		
		\begin{proof}
			The proof is exactly analogous to that of the graded set case.
		\end{proof}

		\begin{definition}
			A \emph{globular operad} is a monoid in $\Coll$ with respect to the monoidal product $\mysquare$.
		\end{definition}

		As with classical operads, defining globular operads in this way is concise but unfortunately condensed and abstract.  Let us briefly unpack this definition to get a feel for why such a structure should be called an operad.  Consider the following globular operad $\mathcal{O}$ given by an underlying collection $o:\mathcal{O} \rightarrow \Tone$ together with the two collection homomorphisms $m: \mathcal{O} \mysquare \mathcal{O} \rightarrow \mathcal{O}$ and $e: \one \rightarrow \mathcal{O}$, called multiplication and unit respectively.  Seeing why this data should intuitively make sense as an operad is analogous to seeing how monoids in $\Grdset$ with respect to the $\mysquare$ tensor product are classical nonsymmetric operads.  As a monoid internal to $\Coll$ these must satisfy the following associativity and unital conditions:
		$$\xymatrix{(\mathcal{O} \mysquare \mathcal{O}) \mysquare \mathcal{O} \ar[rr]^{\alpha} \ar[d]_{m \mysquare \id_{\mathcal{O}}} & & \mathcal{O} \mysquare (\mathcal{O} \mysquare \mathcal{O}) \ar[rr]^{\id_{\mathcal{O}} \mysquare m} & & \mathcal{O} \mysquare \mathcal{O} \ar[d]^{m} \\
		\mathcal{O} \mysquare \mathcal{O} \ar[rrrr]_{m} & & & & \mathcal{O} }$$

		$$\xymatrix{\one \mysquare \mathcal{O} \ar[rr]^{e \mysquare \id_{\mathcal{O}}} \ar[drr]_{\lambda} & & \mathcal{O} \mysquare \mathcal{O} \ar[d]_{m} & & \mathcal{O} \mysquare \one \ar[ll]_{\id_{\mathcal{O}} \mysquare e} \ar[dll]^{\rho} \\
		& & \mathcal{O} & & }$$
		Note then that the morphism $m$ gives us a way of specifying how to compose cells of $\mathcal{O}$ with other cells of $\mathcal{O}$ in a way that is well-defined and associative.  This again mirrors nonsymmetric classical operads which come equipped with a set of operations which may in turn be plugged into each other to produce composite operations.  Furthermore, the morphism $e$ specifies which cells of $\mathcal{O}$ behave as identities of varying dimensions for this operadic composition.  We may even recover the classical notion of a nonsymmetric operad from this more general construction by realizing that a graded set is essentially a collection in which the underlying globular set consists only of 1-cells, each of which has as its boundary the same single 0-cell.

	\section{Exponentials and the internal hom in \Coll}\label{sec-exp-coll}
        Having seen how to construct internal homs in $\Grdset$, we now wish to see the analogous construction in $\Coll$.  In fact, the entire process will be completely analogous to the constructions above for graded sets.  We will proceed by first showing how a morphism of globular sets gives rise to three canonical functors.  We will then, again as a warm up, see how these can be used to construct exponential objects in $\Coll$.  We will then conclude this section by showing how a slight modification of the process for building exponentials gives rise to the process for building the internal hom with respect to the product $\mysquare$ in $\Coll$.
 
		Let $\phi: \mathcal{A} \rightarrow \mathcal{B}$ be a globular set map.  Just as with the construction of classical operads above, we again get an induced functor $\phi^*:\Glob/\mathcal{B} \rightarrow \Glob/\mathcal{A}$ between slice categories called a $\textit{change of base}$ functor which is defined analogously.  It takes a globular set map and returns its pullback along $\phi$.  Once again the functor $\phi^*$ has both a left and right adjoint.  Its left adjoint $\Sigma_{\phi}: \Glob/\mathcal{A} \rightarrow \Glob/\mathcal{B}$ is also composition with $\phi$.  Its right adjoint $\Pi_{\phi}: \Glob/\mathcal{A} \rightarrow \Glob/\mathcal{B}$ is again a bit complicated to describe in general.   More detail on the general construction of $\Pi_{\phi}$ can again be found in \cite{maclane1994sheaves}.  We nonetheless know that such a functor must exist.  We can intuitively think of the fibered globular sets in the image of $\Pi_{\phi}(\xi:\mathcal{X} \rightarrow \mathcal{A})$ as the globular set fibered over $\mathcal{B}$ of generalized sections of the globular set map $\xi$, by analogy to the construction of defining the right adjoint in the graded set case.

		We can then construct the functor $- \times \mathcal{B}: \Coll \rightarrow \Coll$ for a fixed collection $b:\mathcal{B} \rightarrow \Tone$ analogously as before.  It can be written as a composition of the functors defined above.  We get that
		$$- \times \mathcal{B} = \Sigma_{b}b^*$$
		again since $\Coll$ is a slice category.  We can then immediately compute its right adjoint $-^{B}:\Coll \rightarrow \Coll$, which gives the exponentiation by $\mathcal{B}$ functor in $\Coll$, by taking the right adjoint of each factor in the composition and reversing the order in which they are composed, which leads to the following formula:
		$$-^B = \Pi_{b}b^*$$
		Again $b^*(a):\mathcal{A} \pullbackSub{a}{b} \mathcal{B} \rightarrow \mathcal{B}$ is simply second projection.  Thus we see that this functor takes a collection $a:\mathcal{A} \rightarrow \Tone$ and applies the functor $\Pi_{b}$ mentioned above to get an object $\Pi_{b}(b^*(a))$ in $\Coll$ of sections of the globular set projection map $b^*(a)$.  And since this object is in $\Coll$ we can, again by analogy to the graded set case, think of the cells in each fiber $\Gamma_{\sigma}$, for each $\sigma \in \Tone$ as a choice of how to associate cells of shape $\sigma$ from $\mathcal{B}$ to a cell of shape $\sigma$ in $\mathcal{A}$.  In other words, it again defines a map from $\mathcal{B}_{\sigma}$ to $\mathcal{A}_{\sigma}$.  Recall though that all maps of collections carry $\sigma$ fibers to $\sigma$ fibers by their definition as globular sets fibered over $\Tone$.  Thus we can think of the exponential object $\mathcal{A}^{\mathcal{B}}$ as the collection of globular set maps from $\mathcal{B}$ to $\mathcal{A}$ `cut up' into their $\sigma$ fiber restrictions for each $\sigma \in \Tone$.  However, unlike in the case involving graded sets, it is not immediately clear that there is a canonical way to do such a restriction.  But this is precisely the construction being captured by the functor $\Pi_b$.

		Now consider the functor $- \mysquare \mathcal{B}: \Coll \rightarrow \Coll$ for the same collection $b:\mathcal{B} \rightarrow \Tone$.  We will again construct the internal hom with respect to $\mysquare$ using the analogous procedure.  We write $- \mysquare \mathcal{B}$ as the following composition:
		$$- \mysquare \mathcal{B} = \Sigma_{\mu_{\one}}\Sigma_{\mathcal{T}(b)}\mathcal{T}(!_{\mathcal{B}})^*$$
		Note that this functor takes the collection $a:\mathcal{A} \rightarrow \Tone$ to the collection $a \mysquare b: \mathcal{A} \mysquare \mathcal{B} \rightarrow \Tone$, where the arity map $a \mysquare b$ is exactly the image of $\Sigma_{\mu_{\one}}\Sigma_{\mathcal{T}(b)}\mathcal{T}(!_{\mathcal{B}})^*(a)$.  Again as before, this is exactly composition in the augmented pullback diagram used to define the composition tensor product $\mysquare$ in $\Coll$.  Writing the functor $- \mysquare \mathcal{B}$ in this way, just as with $- \times \mathcal{B}$ above, allows us to again compute the appropriate right adjoint $[\mathcal{B},-]:\Coll \rightarrow \Coll$ by taking the right adjoint of each factor in the composition and reversing the order in which they are composed.  This then leads to the following formula:
		$$[\mathcal{B},-] = \Pi_{\mathcal{T}(!_{\mathcal{B}})}\mathcal{T}(b)^*\mu_{\one}^*$$
		We shall again consider first how the composite $\mathcal{T}(b)^*\mu_{\one}^*$ acts on a collection $a:\mathcal{A} \rightarrow \Tone$.  Recall that the map $\mathcal{T}(b)^*\mu_{\one}^*(a)$ is given as the topmost edge in the appropriate double pullback diagram to get the globular set map
		$$\mathcal{T}(b)^*\mu_{\one}^*(a): (\mathcal{A} \pullbackSub{a}{\mu_{\one}^{*}} \mathcal{T}^2(\one)) \pullbackSub{\mu_{\one}^{*}(a)}{\mathcal{T}(b)} \mathcal{T}(\mathcal{B}) \rightarrow \mathcal{T}(\mathcal{B})$$
		which is simply second projection.  We can intuitively think of this map as associating to each cell $\beta \in \mathcal{T}(\mathcal{B})$, which is named by a pasting diagram labeled by cells in $\mathcal{B}$, a pair $(\alpha, t) \in \mathcal{A} \times \mathcal{T}(\Tone)$ consisting of cell $\alpha \in \mathcal{A}$ and cell of cells $t$ of shape $\sigma \in \Tone$ (i.e. a globular word of cells in $\Tone$ indexed by the diagram of shape $\sigma$) such that the shape of $\alpha$ is the same as the shape of the cell obtained by gluing together the cells in $t$ per the pasting formula given by $\sigma$.  Moreover, the unlabeled cells of $t$ each have the same shape as the corresponding cells which make up the labeled diagram $\beta$.  We then apply $\Pi_{\mathcal{T}(!_{\mathcal{B}})}$ to $\mathcal{T}(b)^*\mu_{\one}^*(a)$ to get the desired internal hom.  

		Via this construction, we can now compute the internal hom $\mathcal{H}_{\mathcal{B},\mathcal{A}}: [\mathcal{B},\mathcal{A}] \rightarrow \Tone$ in $\Coll$ between any collections $b:\mathcal{B} \rightarrow \Tone$ and $a:\mathcal{A} \rightarrow \Tone$.  Although there are some subtle technical differences from the graded set case which are lurking in the construction encoded by the functor $\Pi_{\mathcal{T}(!_{\mathcal{B}})}$, we can intuitively think of cells in each fiber $[A,B]_{\sigma}$ of our internal hom in the following way.  Recall that the internal hom is constructed as the object of general sections of the globular set map $\mathcal{T}(b)^*\mu_{\one}^*(a)$ defined above.  Moreover, a cell $\beta \in \mathcal{T}(\mathcal{B})_{\sigma}$ can be thought of as a choice of cells $\{\beta_{\tau}\}_{\tau \in \sigma}$ from $\mathcal{B}$ glued together along halves of their boundaries as prescribed by the pasting formula given by the pasting diagram $\sigma$.  Or rather, we can think of them as a coloring of the diagram $\sigma$ by cells in $\mathcal{B}$.  This allows us to think of a cell $\gamma \in [\mathcal{B},\mathcal{A}]_{\sigma}$ as a choice of a cell $\alpha \in \mathcal{A}$ to correspond to each coloring of the diagram $\sigma$ by cells of $\mathcal{B}$ so that the shape of $\alpha$ is the same as the shape of the diagram obtained by gluing the cells $\{\beta_{\tau}\}_{\tau \in \sigma}$ together via the pasting formula given by $\sigma$.  In other words, a `map' in the internal hom is roughly a thing that takes a coloring of the diagram $\sigma \in \Tone$ by cells from the source and picks a cell of the target that has the same arity shape as the cells from the source after all the pasting compositions prescribed by the diagram $\sigma$ have been performed.

	\section{The globular tautological operad}\label{sec-glob-taut}
		Consider the collection $x:\mathcal{X} \rightarrow \Tone$.  We shall now construct the tautological operad on $\mathcal{X}$, denoted $Taut(\mathcal{X})$, analogously to the construction in the case of graded sets.  We again define $Taut(\mathcal{X}):= [\mathcal{X},\mathcal{X}]$ via the internal hom construction in $\Coll$.  The underlying collection for the tautological globular operad on $\mathcal{X}$ can be thought of as abstractly encoding all the possible operations that take a coloring of a pasting diagram of shape $\sigma \in \Tone$ by globular cells from $\mathcal{X}$ to a single globular cell from $\mathcal{X}$ whose shape is the same as the `word of cells' after the each of the pasting compositions prescribed by $\sigma$ are evaluated to give a composed cell in $\mathcal{X}$.  But since they are constructed using the internal hom, rather than the set valued hom, these `maps' from $\mathcal{X}$ to $\mathcal{X}$ naturally fiber over $\Tone$ so that we can again place a canonical operad structure on $Taut(\mathcal{X})$.  The operad identity is given by the map $\iota: \one \rightarrow [\mathcal{X},\mathcal{X}]$ which maps each single $k$-cell of $\one$ to the the respective $k$-cell of $[\mathcal{X},\mathcal{X}]$  which corresponds to the identity operation on $k$-cells of $\mathcal{X}$.  This map $\iota$ can be constructed canonically as the currying of the left unitor $\lambda_{\mathcal{X}}: \one \mysquare \mathcal{X} \rightarrow \mathcal{X}$ for the monoidal structure in $\Coll$.  The composition map $\nu:[\mathcal{X},\mathcal{X}] \mysquare [\mathcal{X},\mathcal{X}] \rightarrow [\mathcal{X},\mathcal{X}]$ is the canonical map which takes a pair $(a,w) \in [\mathcal{X},\mathcal{X}] \mysquare [\mathcal{X},\mathcal{X}]$ and composes each of the letters from the word $w \in \mathcal{T}([\mathcal{X},\mathcal{X}])$ with each of the respective inputs for the operation $a \in [\mathcal{X},\mathcal{X}]$.  It can be canonically constructed as follows.  Consider the counit $\epsilon^{\mathcal{A}}: [\mathcal{A},-] \mysquare \mathcal{A} \Rightarrow \id_{\Coll}$ of the hom-tensor adjunction in $\Coll$ between $- \mysquare \mathcal{A}$ and $[\mathcal{A},-]$, which has components
		$\epsilon^{\mathcal{B}}_{\mathcal{A}}:[\mathcal{A},\mathcal{B}]\mysquare \mathcal{A} \rightarrow \mathcal{B}$ for collection $\mathcal{A}$.  We then get a map
		$$\mathcal{K}: [\mathcal{X},\mathcal{X}] \mysquare ([\mathcal{X},\mathcal{X}] \mysquare \mathcal{X}) \rightarrow [\mathcal{X},\mathcal{X}] \mysquare \mathcal{X} \rightarrow \mathcal{X}$$
		which is the composite $\mathcal{K} := \epsilon^{\mathcal{X}}_{\mathcal{X}}(\id_{\mathcal{X}} \mysquare \epsilon^{\mathcal{X}}_{\mathcal{X}})$.  The operad multiplication for $[\mathcal{X},\mathcal{X}]$ is then the currying of the map $\mathcal{K}$.

		\begin{theorem}
			Given a collection $x: \mathcal{X} \rightarrow \Tone$, collection Taut($\mathcal{X}):[\mathcal{X},\mathcal{X}] \rightarrow \Tone$ admits the structure of a globular operad.
		\end{theorem}
		
		\begin{proof}
			We need only to show that for $Taut(\mathcal{X}):[\mathcal{X},\mathcal{X}] \rightarrow \Tone$ the collection morphisms $\iota: \one \rightarrow [\mathcal{X},\mathcal{X}]$ and $\nu: [\mathcal{X},\mathcal{X}] \mysquare [\mathcal{X},\mathcal{X}] \rightarrow [\mathcal{X},\mathcal{X}]$ satisfy the commutative diagrams required of a monoid object in $\Coll$.  This can be seen, just as in the graded set case, by first currying the maps in the relevant diagrams and checking to see that these new curred diagrams, whose commutativity is equivalent with that of the originals, do in fact commute.  The details of which are the same as those from the graded set case above.
		\end{proof}

        \begin{definition}
            A collection $\mathcal{X}$ is called a $\emph{degenerate collection}$ if the arity map factors as $x = [id] \circ !_{\mathcal{X}}$, where $[id]:\one \rightarrow \Tone$ is the globular set map defined above.
        \end{definition}

        Note that any globular set can be canonically associated to a particular degenerate collection by sending each $n$-cell to the unique $n$-cell in $\Tone$ that's in the image of the map $[id]$. This is analogous to canonically associating to each set a degenerate graded set concentrated over zero.  And this once again allows us to concisely define algebras for a globular operad.
        
		\begin{definition}
			Let $o:\mathcal{O} \rightarrow \Tone$ be a globular operad.  A $\emph{left }\mathcal{O}\emph{-module}$ is a collection homomorphism $f:\mathcal{O} \rightarrow Taut(\mathcal{X})$ for some collection $x:\mathcal{X} \rightarrow \Tone$.  An $\mathcal{O}\emph{-algebra}$ in $\Glob$ is a left $\mathcal{O}$-module such that the collection $\mathcal{X}$ is degenerate.
		\end{definition}

		We again say that in such a case the operad $\mathcal{O}$ acts on the collection $\mathcal{X}$.  Moreover, when $x = [id] \circ !_{\mathcal{X}}: \mathcal{X} \rightarrow \Tone$, the operad is acting on what we can think of as a globular set in disguise because each globular cell in $\mathcal{X}$ sits above one of the special identity cells in $\Tone$ described above. Note that because of how $\mysquare$ is defined on collections, the arity of a composable pair is the arity shape of the original cell expanded to include the shapes of the cells which were plugged in to the original cell.  But when we compose with collection cells that sit over one of these special identities, the arity of the composable pair does not expand in this typical way.  Instead, such composable pairs have arities that also sit over one of these special identities of the appropriate dimension.  It is analogous to the graded set case in which the cells concentrated in degree zero, when composed as graded set elements via $\mysquare$, have a composite arity of a sum of zeros.  Hence, the arity of everything in the graded set concentrated in degree zero remains zero both before and after composition.  Similarly, in a degenerate collections the arity of any cell in $\mathcal{X}$ is simply the iterative identity of the appropriate dimension, both before and after applying any well-defined pasting operation.
		
		Note then that in \cite{leinster2004higher} Leinster defines, just as with classical operads, an algebra $\mathcal{X}$ for a globular operad $\mathcal{O}$ as an algebra for the canonical monad $\mathcal{T}_{\mathcal{O}}$ associated to the operad $\mathcal{O}$.  Let $\mathcal{O}(n)$ be the globular set of $n$-cells in $\mathcal{O}$, $\mathcal{O}(\pi)$ be the globular set of cells whose arity is of shape $\pi$, and $\hat{\pi}$ the cell $\pi$ thought of as a globular sub-object of $\mathcal{X}$.  Then Leinster's definition amounts to equipping the globular set $\mathcal{X}$ with, for each $n \in \mathbb{N}$ and every $\pi \in \mathcal{O}(\pi)$, functions $h_{\pi}: \mathcal{O}(\pi) \times \Glob(\hat{\pi},\mathcal{X}) \rightarrow \mathcal{X}(n)$.  For each cell $\theta \in \mathcal{O}$, Leinster then defines an induced map $\overline{\theta}(-) := h_{\pi}(\theta, -):\Glob \rightarrow \Glob$ which takes a globular set map (i.e. a labeling of the diagram shape $\pi$ by cells of the globular set $\mathcal{X}$) and returns an $n$-cell from $\mathcal{X}$.  At first this approach may seem quite different from the definition of algebras given above.  However, we have the following result.

		\begin{theorem}
			An algebra $\mathcal{X}$ for the canonical monad $\mathcal{T}_{\mathcal{O}}$ associated to the globular operad $\mathcal{O}$ are algebras in $\Glob$ for $\mathcal{O}$ as specified by an operad homomorphism to the tautological globular operad $Taut(\mathcal{X})$.
		\end{theorem}

		\begin{proof}
			Let the globular set $\mathcal{X}$ be an algebra for the canonical monad $\mathcal{T}_{\mathcal{O}}$ associated to the globular operad $\mathcal{O}$ as described by Leinster.  Realize that each of these maps are can be thought of as the components, indexed by $\Tone$, of the collection homomorphism $F:\mathcal{O} \mysquare \mathcal{X} \rightarrow \mathcal{X}$ when $\mathcal{X}$ is thought of a degenerate collection.  Each of these collection maps can then be curried via the hom-tensor adjunction in $\Coll$ with respect to the $\mysquare$ tensor product to get collection maps $f: \mathcal{O} \rightarrow [\mathcal{X},\mathcal{X}] = Taut(\mathcal{X})$.  This homomorphism can then be considered as a family of maps $F_{\pi}:\mathcal{O}(\pi) \times \Glob(\hat{\pi},\mathcal{X}) \rightarrow \mathcal{X}(n)$ for each $\pi \in \mathcal{O}(n)$ with $n \in \mathcal{O}(n)$.  This is because $\mathcal{O}(\pi)$ is the restriction of $\mathcal{O}$ to the fiber over $\pi \in \Tone$, $\Glob(\hat{\pi},\mathcal{X})$ is the set of colorings of the cell $\pi \in \Tone$ by cells of $\mathcal{X}$, and for each $\theta \in \mathcal{O}(\pi)$, the induced map $\overline{\theta}(-) = h_{\pi}(\theta, -):\Glob \rightarrow \Glob$ obtained by fixing the first entry of $F_{\pi}$ takes a coloring of $\pi$, thought of as a globular set, by cells of $\mathcal{X}$ and gives an $n$-cell in $\mathcal{X}$.  But this is exactly what $F$ does if we restrict it's first input.  Given a cell $\theta$ of arity $\pi \in \mathcal{O}$, $F$ induces a map $\underline{\theta} = F_{\pi}(\theta, -)$ which takes a cell $\gamma$ of $\mathcal{X}$, which thought of as a degenerate collection in the second factor of a $\mysquare$ product means that it is a globular word in $\mathcal{X}$ that can be composed with $\theta$ (or equivalently, a coloring of $\theta$ by elements of $\mathcal{X}$), and gives another element of $\mathcal{X}$ which is the composition of $\theta$ and $\gamma$ in the algebra specified by $f$.   Hence it follows that algebras as we have described them are exactly the same algebras specified by Leinster's construction, merely presented in a different way. 
		\end{proof}
	
		We shall conclude this exposition with the following observation, which follows immediately from thinking of algebras for globular operads in the representation form described above.
		
		\begin{theorem}
			An algebra for a globular operad $\mathcal{O}$ is an algebra for every globular operad $\mathcal{P}$ which maps to $\mathcal{O}$.  In particular, an algebra for $\mathcal{O}$ is an algebra for every globular sub-operad.
		\end{theorem}

		\begin{proof}
			Recall that any globular operad $\mathcal{O}$ is a monoid in $\Coll$ and an algebra for such an operad is specified by a morphism of globular sets $\phi: \mathcal{O} \rightarrow Taut(\mathcal{X})$.  Let $\mathcal{P}$ be another globular operad and $\psi: \mathcal{P} \rightarrow \mathcal{O}$ be a morphism of globular sets.  Then the map $\phi (\psi): \mathcal{P} \rightarrow Taut(\mathcal{X})$ induces on $\mathcal{X}$ the structure of an $\mathcal{P}$-algebra.  Moreover, as any globular sub-operad $\tilde{\mathcal{O}}$ of $\mathcal{O}$ comes from removing a certain subset of cells from $\mathcal{O}$, the algebra for $\mathcal{O}$ specified by the map $\phi$ is also an algebra for $\tilde{O}$ simply by restricting $\phi$ to $\tilde{\mathcal{O}}$.  Moreover, any further subsets of $\tilde{\mathcal{O}}$ corresponds to a further restriction of $\phi$, showing that any algebra for $\mathcal{O}$ is an algebra for every globular sub-operad $\tilde{\mathcal{O}}$. 
		\end{proof}

\nocite{*}
\bibliographystyle{abbrv}
\bibliography{OnTautGlobOp}

\end{document}